\documentclass[12pt,reqno]{article}

\usepackage[square, comma, sort&compress, numbers]{natbib}
\usepackage{amsmath,amsthm,amsfonts,amssymb,latexsym,mathrsfs,color}

\newtheorem{thm}{Theorem}[section]
\newtheorem{lem}[thm]{Lemma}

\newtheorem{coro}[thm]{Corollary}
\newtheorem{conj}[thm]{Conjecture}

\theoremstyle{definition}
\newtheorem{exm}{Example}[section]
\newtheorem{rem}{Remark}[section]

\numberwithin{equation}{section}

\title{Total positivity of recursive matrices}
\author{Xi Chen, Huyile Liang and Yi Wang
\thanks{{\it Email address:}\quad wangyi@dlut.edu.cn (Y. Wang)}}
\date{\footnotesize School of Mathematical Sciences, Dalian University of Technology, Dalian 116024, PR China}

\begin{document}

\maketitle

\begin{abstract}
Let $A=[a_{n,k}]_{n,k\ge 0}$ be an infinite lower triangular matrix defined by the recurrence
$$a_{0,0}=1,\quad a_{n+1,k}=r_{k}a_{n,k-1}+s_ka_{n,k}+t_{k+1}a_{n,k+1},$$
where $a_{n,k}=0$ unless $n\ge k\ge 0$ and $r_k,s_k,t_k$ are all nonnegative.
Many well-known combinatorial triangles are such matrices,
including the Pascal triangle, the Stirling triangle (of the second kind), the Bell triangle,
the Catalan triangles of Aigner and Shapiro.
We present some sufficient conditions such that the recursive matrix $A$ is totally positive.
As applications we give the total positivity of the above mentioned combinatorial triangles in a unified approach.
\bigskip
\\[1pt]
\end{abstract}

\section{Introduction}
\hspace*{\parindent}

Let $A=[a_{n,k}]_{n,k\ge 0}$ be an infinite matrix.
It is called {\it totally positive of order $r$} (or shortly, {\it TP$_r$}),
if its minors of all orders $\le r$ are nonnegative.
It is called {\it TP} if its minors of all orders are nonnegative.
Let $(a_n)_{n\ge 0}$ be an infinite sequence of nonnegative numbers.
It is called a {\it P\'olya frequency sequence of order $r$} (or shortly, a {\it PF$_r$} sequence),
if its Toeplitz matrix
$$[a_{i-j}]_{i,j\ge 0}=\left[
\begin{array}{lllll}
a_{0} &  &  &  &\\
a_{1} & a_{0} & & &\\
a_{2} & a_{1} & a_{0} &  &\\
a_{3} & a_{2} & a_{1} & a_{0} &\\
\vdots & &\cdots & & \ddots\\
\end{array}
\right]$$
is TP$_r$.
It is called {\it PF} if its Toeplitz matrix is TP.
We say that a finite sequence $a_0,a_1,\ldots, a_n$ is PF$_r$ (PF, resp.)
if the corresponding infinite sequence $a_0,a_1,\ldots,a_n,0,\ldots$ is PF$_r$ (PF, resp.).
We say that a nonnegative sequence $(a_n)$ is {\it log-convex} ({\it log-concave}, resp.)
if $a_{i}a_{j+1}\ge a_{i+1}a_{j}$ ($a_{i}a_{j+1}\le a_{i+1}a_{j}$, resp.) for $0\le i<j$.
Clearly, the sequence $(a_n)$ is log-concave if and only if it is PF$_2$,
i.e., its Toeplitz matrix $[a_{i-j}]_{i,j\ge 0}$ is TP$_2$,
and the sequence is log-convex if and only if its Hankel matrix
$$[a_{i+j}]_{i,j\ge 0}=\left[
\begin{array}{lllll}
a_{0} & a_1 & a_2 & a_3 & \cdots\\
a_{1} & a_{2} & a_3 & a_4 & \cdots\\
a_{2} & a_{3} & a_{4} & a_5 & \cdots\\
a_{3} & a_{4} & a_{5} & a_{6} & \cdots\\
\vdots & \vdots &\vdots & \vdots & \ddots\\
\end{array}
\right]$$
is TP$_2$~\cite{Bre89}.

Let $\pi=(r_k)_{k\ge 1}, \sigma=(s_k)_{k\ge 0}, \tau=(t_k)_{k\ge 1}$ be three sequences of nonnegative numbers
and define an infinite lower triangular matrix
$$A:=A^{\pi,\sigma,\tau}=[a_{n,k}]_{n,k\ge 0}
=\left[\begin{array}{lllll}
a_{0,0} &  &  &  &\\
a_{1,0} & a_{1,1} & & &\\
a_{2,0} & a_{2,1} & a_{2,0} &  &\\
a_{3,0} & a_{3,1} & a_{3,2} & a_{3,3}\\
\vdots & & & & \ddots\\
\end{array}\right]$$
by the recurrence
\begin{equation}\label{rst-eq}
a_{0,0}=1,\quad a_{n+1,0}=s_0a_{n,0}+t_1a_{n,1},\quad a_{n+1,k}=r_{k}a_{n,k-1}+s_ka_{n,k}+t_{k+1}a_{n,k+1},
\end{equation}
where $a_{n,k}=0$ unless $n\ge k\ge 0$.
Following Aigner~\cite{Aig01},
we say that $A^{\pi,\sigma,\tau}$ is the {\it recursive matrix}
and $a_{n,0}$ are the {\it Catalan-like numbers}
corresponding to $(\pi,\sigma,\tau)$.
Such triangles arise often in combinatorics
and many well-known counting coefficients are the Catalan-like numbers.
The following are several basic examples of recursive matrices.
\begin{exm}
\begin{itemize}
  \item [\rm (i)]   
  The Pascal triangle $P=\left[\binom{n}{k}\right]_{n,k\ge 0}$ satisfies $\binom{n+1}{k}=\binom{n}{k-1}+\binom{n}{k}$.
  \item [\rm (ii)]
  The Stirling triangle (of the second kind) $S=[S(n,k)]_{n,k\ge 0}$ satisfies $S(n+1,k)=S(n,k-1)+(k+1)S(n,k)$.
  \item [\rm (iii)] 
  The Catalan triangle of Aigner is
  $$C=[C_{n,k}]=\left[
      \begin{array}{rrrrrr}
        1 &  &  &  &  &  \\
        1 & 1 &   &   &   &   \\
        2 & 3 & 1 &   &   &   \\
        5 & 9 & 5 & 1 &   &   \\
        14 & 28 & 20 & 7 & 1 &   \\
        \vdots &  & &  &  & \ddots \\
      \end{array}
    \right],$$
  where $C_{n+1,0}=C_{n,0}+C_{n,1}, C_{n+1,k}=C_{n,k-1}+2C_{n,k}+C_{n,k+1}$~\cite{Aig99}.
  The corresponding Catalan-like numbers $C_{n,0}$ are precisely the Catalan numbers $C_n$.
  
  The Catalan triangle of Shaprio is
  \begin{equation*}\label{ct}
  B=[B_{n,k}]=\left[
      \begin{array}{rrrrrr}
        1 &  &  &  &  &  \\
        2 & 1 &   &   &   &   \\
        5 & 4 & 1 &   &   &   \\
        14 & 14 & 6 & 1 &   &   \\
        42 & 48 & 27 & 8 & 1 &   \\
        \vdots &  & &  &  & \ddots \\
      \end{array}
    \right],
  \end{equation*}
  where $B_{n+1,k}=B_{n,k-1}+2B_{n,k}+B_{n,k+1}$~\cite{Sha76}.
  The corresponding Catalan-like numbers $B_{n,0}$ are precisely the Catalan numbers $C_{n+1}$.
  There are a lot of papers to consider combinatorics of the Catalan triangle~\cite{Aig99,Aig08,He13,Rog78,Sha76,SGWW91}.
  See also Sloane's OEIS~\cite[A039598]{OEIS}.
  \item [\rm (iv)] 
  The Bell triangle, introduced by Aigner~\cite{Aig99b}, is
  $$X=[X_{n,k}]=\left[
      \begin{array}{rrrrrr}
        1 &  &  &  &  &  \\
        1 & 1 &   &   &   &   \\
        2 & 3 & 1 &   &   &   \\
        5 & 10 & 6 & 1 &   &   \\
        15 & 37 & 31 & 10 & 1 &   \\
        \vdots &  & &  &  & \ddots \\
      \end{array}
    \right],$$
  where $X_{n+1,k}=X_{n,k-1}+(k+1)X_{n,k}+(k+1)X_{n,k+1}$.
  The corresponding Catalan-like numbers $X_{n,0}$ are the Bell numbers $B_n$.
\end{itemize}
\end{exm}

Aigner~\cite{Aig99,Aig99b,Aig01,Aig08} studied various combinatorial properties of recursive matrices
and Hankel matrices of the Catalan-like numbers.
It is well known that the Pascal triangle is TP~\cite[p. 137]{Kar68}.
Brenti~\cite{Bre95} showed, among other things, that the Stirling triangle is TP.
Very recently,
Zhu~\cite[Theorem 3.1]{Zhu13} showed that if $s_{k-1}s_{k}\ge r_kt_{k}$ for $k\ge 1$,
then the sequence $(a_{n,0})_{n\ge 0}$ of Catalan-like numbers defined by (\ref{rst-eq}) is log-convex.
Zhu~\cite[Theorem 2.1]{Zhu14} also showed that
if $r_k,s_k$ are nonnegative quadratic polynomials in $k$ and $t_k=0$ for all $k$,
then the corresponding matrix $A$ is TP.
The object of this paper is to give some sufficient conditions
for total positivity of recursive matrices.
In the next section,
we present our main results.
As applications, we show that many well-known combinatorial triangles,
including the Pascal triangle, the Stirling triangle, the Bell triangle, the Catalan triangles of Aigner and Shapiro
are TP in a certain unified approach.
As consequences, the corresponding Catalan-like numbers,
including the Catalan numbers and the Bell numbers, form a log-convex sequence respectively.
In Section 3, we point out that our results can be carried over verbatim to their $q$-analogue.
We also propose a couple of problems for further work.

\section{Main results and applications}
\hspace*{\parindent}

We first review some basic facts about TP matrices.
The first is direct by definition and
the second follows immediately from the classic Cauchy-Binet formula.

\begin{lem}\label{lps-lem}
A matrix is TP$_r$ (TP, resp.) if and only if
its leading principal submatrices are all TP$_r$ (TP, resp.).
\end{lem}

\begin{lem}\label{prod-lem}
The product of two TP$_r$ (TP, resp.) matrices is still TP$_r$ (TP, resp.).
\end{lem}

Rewrite the recursive relation (\ref{rst-eq}) as
$$
\left[
\begin{array}{ccccc}
a_{1,0} & a_{1,1} &  &  &  \\
a_{2,0} & a_{2,1} & a_{2,2} &  &  \\
a_{3,0} & a_{3,1} & a_{3,2} & a_{3,3} & \\
& & \cdots  & & \ddots   \\
\end{array}
\right]
=
\left[
\begin{array}{cccc}
a_{0,0} &  &  &  \\
a_{1,0} & a_{1,1} &  &  \\
a_{2,0} & a_{2,1} & a_{2,2} &  \\
& \cdots &  & \ddots   \\
\end{array}
\right]
\left[
\begin{array}{cccc}
s_0 & r_1 &  &\\
t_1 & s_1 & r_2 &\\
 & t_2 & s_2 & \ddots\\
& &\ddots & \ddots \\
\end{array}\right],
$$
or briefly,
\begin{equation}\label{AJ}
\overline{A}=AJ
\end{equation}
where $\overline{A}$ is obtained from $A$ by deleting the $0$th row and $J$ is the Jacobi matrix
\begin{equation}\label{J-eq}
J:=J^{\pi,\sigma,\tau}=\left[
\begin{array}{ccccc}
s_0 & r_1 &  &  &\\
t_1 & s_1 & r_2 &\\
 & t_2 & s_2 & r_3 &\\
 & & t_3 & s_3 & \ddots\\
& & &\ddots & \ddots \\
\end{array}\right].
\end{equation}
Clearly, the recursive relation (\ref{rst-eq}) is decided completely by the tridiagonal matrix $J$.
Call $J$ the {\it coefficient matrix} of the recursive relation (\ref{rst-eq}).
For convenience, we also call $J$ the {\it coefficient matrix} of the recursive matrix $A$.

For example, the coefficient matrix of the Bell triangle is
  $$\left[\begin{array}{ccccc}
  1 & 1 &  &  &\\
  1 & 2 & 1 &\\
  & 2 & 3 & 1 &\\
  & & 3 & 4 & \ddots\\
  & & &\ddots & \ddots \\
  \end{array}\right],$$
the coefficient matrices of Catalan triangles of Aigner and Shapiro are
  $$\left[\begin{array}{ccccc}
  1 & 1 &  &  &\\
  1 & 2 & 1 &\\
  & 1 & 2 & 1 &\\
  & & 1 & 2 & \ddots\\
  & & &\ddots & \ddots \\
  \end{array}\right]\quad \textrm{and}\quad
  \left[\begin{array}{ccccc}
  2 & 1 &  &  &\\
  1 & 2 & 1 &\\
  & 1 & 2 & 1 &\\
  & & 1 & 2 & \ddots\\
  & & &\ddots & \ddots \\
  \end{array}\right].$$

\begin{thm}\label{basic-thm}
Let $A$ be a recursive matrix with the coefficient matrix $J$.
\begin{itemize}
  \item [\rm (i)] If $J$ is TP$_r$ (TP, resp.), then so is $A$.
  \item [\rm (ii)] If $A$ is TP$_2$,
  then the sequence $(a_{n,0})_{n\ge 0}$ of the Catalan-like numbers is log-convex.
\end{itemize}
\end{thm}
\begin{proof}
(i)\quad
Clearly, it suffices to consider the TP$_r$ case.
Let
$$
A_n=
\left[\begin{array}{cccc}
a_{0,0} &  &  &\\
a_{1,0} & a_{1,1} &  &\\
\vdots & \vdots & \ddots  & \\
a_{n,0}& a_{n,1} &\cdots& a_{n,n} \\
\end{array}\right],\qquad
J_n=
\left[\begin{array}{cccc}
s_0 & r_1 &  &  \\
t_1 & s_1 & \ddots & \\
&\ddots & \ddots  &  r_n  \\
& & t_n & s_n \\
\end{array}\right]$$
and
$$\overline{A}_{n+1}=
\left[\begin{array}{ccccc}
a_{1,0} & a_{1,1} &  &  &\\
a_{2,0} & a_{2,1} & a_{2,2} &  &\\
\cdots & \cdots & \cdots & \ddots  & \\
a_{n,0} & a_{n,1} & a_{n,2} & \cdots & a_{n,n}\\
a_{n+1,0} & a_{n+1,1} & a_{n+1,2} &\cdots & a_{n+1,n}\\
\end{array}\right]$$
be the $n$th leading principal submatrices of $A, J$ and $\overline{A}$ respectively.
Then $\overline{A}_{n+1}=A_nJ_n$ by (\ref{AJ}).
Now $J$ is TP$_r$, so is $J_n$.
Assume that $A_n$ is TP$_r$.
Then the product $\overline{A}_{n+1}=A_nJ_n$ is also TP$_r$.
It follows that $A_{n+1}$ is TP$_r$.
Thus $A$ is TP$_r$ by induction.

(ii)\quad
By (\ref{rst-eq}), we have
\begin{eqnarray}\label{t2-dec}
\left[\begin{array}{cc}
a_{0,0} & a_{1,0}\\
a_{1,0} & a_{2,0}\\
a_{2,0} & a_{3,0}\\
\vdots & \vdots\\
\end{array}
\right]=
\left[
\begin{array}{ccccc}
a_{0,0} & & & & \\
a_{1,0} & a_{1,1} &  &  &  \\
a_{2,0} & a_{2,1} & a_{2,2} &  &  \\
& & \cdots  & & \ddots   \\
\end{array}
\right]
\left[\begin{array}{cc}
1 & s_0\\
0 & t_1\\
0 & 0\\
\vdots & \vdots\\
\end{array}
\right].
\end{eqnarray}
Clearly, the second matrix in the right hand side of (\ref{t2-dec}) is TP$_2$ since $s_0$ and $t_1$ are nonnegative.
If $A$ is TP$_2$, then so is the matrix in the left hand side of (\ref{t2-dec}),
which is equivalent to the log-convexity of the sequence $(a_{n,0})_{n\ge 0}$.
This completes the proof.
\end{proof}

So we may focus our attention on the total positivity of tridiagonal matrices.
We first give two simple applications of Theorem \ref{basic-thm} from this point of view.

\begin{coro}[{\cite[Theorem 3.1]{Zhu13}}]
If $s_{k-1}s_{k}\ge r_kt_k$ for $k\ge 1$,
then the sequence $(a_{n,0})_{n\ge 0}$ of Catalan-like numbers defined by (\ref{rst-eq}) is log-convex.
\end{coro}
\begin{proof}
If $s_{k-1}s_{k}\ge r_kt_k$ for $k\ge 1$, then $J$ is TP$_2$, and so is $A$ by Theorem \ref{basic-thm}~(i).
Thus $(a_{n,0})_{n\ge 0}$ is log-convex by Theorem \ref{basic-thm}~(ii).
\end{proof}

\begin{coro}\label{rs-thm}
Let $A=[a_{n,k}]_{n,k\ge 0}$ be a recursive matrix defined by
\begin{equation}\label{rs-eq}
a_{0,0}=1,\quad a_{n+1,k}=r_ka_{n,k-1}+s_ka_{n,k}.
\end{equation}
If $r_k$ and $s_k$ are nonnegative,
then $A$ is TP.
\end{coro}
\begin{proof}
In this case, the coefficient matrix is a bidiagonal matrix,
which is obviously TP,
and so is the recursive matrix by Theorem \ref{basic-thm}~(i).
\end{proof}
\begin{rem}
An immediate consequence of Corollary \ref{rs-thm} is Zhu's result~\cite[Theorem 2.1]{Zhu14},
which states that if $r_k,s_k$ are nonnegative quadratic polynomials in $k$,
then the matrix $[a_{n,k}]_{n,k\ge 0}$ defined by (\ref{rs-eq}) is TP.
In particular,
the Pascal triangle and the Stirling triangle are TP.
\end{rem}

There are many well-known results about the total positivity of tridiagonal matrices.
The following is one of them.

\begin{lem}[{\cite[Theorem 4.3]{Pin10}}]\label{pn>0}
A finite nonnegative tridiagonal matrix is TP if and only if
all its principal minors containing consecutive rows and columns are nonnegative.
\end{lem}

Actually, it is also known that
an irreducible nonnegative tridiagonal matrix is TP if and only if all its leading principal minors are positive
\cite[Example 2.2]{Min88}.

We next consider the problem in which case a tridiagonal matrix has nonnegative determinant.
Let $M=[m_{ij}]_{1\le i,j\le n}$ be a real $n\times n$ matrix.
We say that $M$ is {\it row diagonally dominant} if
\begin{equation}\label{dd-ineq}
m_{ii}\ge |m_{i,1}|+\cdots+|m_{i,i-1}|+|m_{i,i+1}|+\cdots+|m_{i,n}|,\quad i=1,2,\ldots,n.
\end{equation}
If all inequalities in (\ref{dd-ineq}) are strict, then we say that $M$ is {\it strictly row diagonally dominant}.
It is well known \cite{Tau49} that
if $M$ is strictly row diagonally dominant, then $|M|>0$.
Moreover, if $M$ is irreducible row diagonally dominant and
there is at least one strict inequality in (\ref{dd-ineq}), then $|M|>0$.
The case for nonnegative tridiagonal matrices is simpler.

\begin{lem}\label{Jn>0}
Let
\begin{equation*}\label{Jn-mat}
J_n=\left[\begin{array}{cccccc}
y_0 & x_1 &  &  &  &\\
z_1 & y_1 & x_2 & &\\
 & z_2 & y_2 & x_3 &  &\\
&  &\ddots & \ddots  &  \ddots &  \\
& & & z_{n-1} & y_{n-1} & x_n\\
&& & & z_n & y_n \\
\end{array}\right],
\end{equation*}
where $x_k,y_k,z_k$ are all nonnegative.
\begin{itemize}
  \item [\rm (i)] If $J_n$ is row diagonally dominant, then $|J_n|\ge 0$.
  \item [\rm (ii)] If $J_n$ is column diagonally dominant, then $|J_n|\ge 0$.
\end{itemize}
\end{lem}
\begin{proof}
(i)\quad
We proceed by induction on $n$.
Assume that $y_n=z_n$. Then
$$|J_n|
=\left|\begin{array}{ccccc}
y_0 & x_1 &  &  &  \\
z_1 & y_1 & x_2 & \\
 & \ddots & \ddots & \ddots &  \\
&  &z_{n-1} & y_{n-1}-x_n  &  x_n  \\
& & & 0 & y_n \\
\end{array}\right|
=y_n\left|\begin{array}{ccccc}
y_0 & x_1 &  & &\\
z_1 & y_1 & x_2 & &\\
& \ddots & \ddots & \ddots &\\
& & z_{n-2} & y_{n-2} & x_{n-1}\\
& & &z_{n-1} & y_{n-1}-x_n\\
\end{array}\right|.$$
Thus $|J_n|$ is nonnegative by the inductive hypothesis.
Assume that $y_n>z_n$. Then
$$|J_n|=\left|\begin{array}{ccccc}
y_0 & x_1 &  &  &  \\
z_1 & y_1 & x_2 & \\
 & \ddots & \ddots & \ddots &  \\
&  &z_{n-1} & y_{n-1} &  x_n  \\
& & & z_n & z_n \\
\end{array}\right|
+\left|\begin{array}{ccccc}
y_0 & x_1 &  &  &  \\
z_1 & y_1 & x_2 & \\
 & \ddots & \ddots & \ddots &  \\
&  &z_{n-1} & y_{n-1} &  x_n  \\
& & & 0 & y_n-z_n \\
\end{array}\right|.$$
Clearly, two determinants on the right hand side are nonnegative,
so is $|J_n|$.

(ii)\quad
Apply (i) to the transpose $J_n^T$ of $J_n$.
\end{proof}

Combining Theorem \ref{basic-thm}~(i) and Lemma~\ref{Jn>0} we obtain the following criterion.

\begin{thm}\label{dd-thm}
Let $A$ be the recursive matrix defined by (\ref{rst-eq}).
\begin{itemize}
  \item [\rm (i)]
  If $s_0\ge r_1$ and $s_k\ge r_{k+1}+t_{k}$ for $k\ge 1$, then $A$ is TP.
  \item [\rm (ii)]
  If $s_0\ge t_1$ and $s_k\ge r_k+t_{k+1}$ for $k\ge 1$, then $A$ is TP.
\end{itemize}
\end{thm}

\begin{thm}\label{rst-thm}
Let $A$ be the recursive matrix defined by (\ref{rst-eq}).
If $s_0\ge 1$ and $s_k\ge r_kt_k+1$ for $k\ge 1$,
then $A$ is TP.
\end{thm}
\begin{proof}
By Theorem~\ref{basic-thm},
we need to show that the corresponding coefficient matrix $J$ is TP.
By Lemma~\ref{pn>0},
it suffices to show that the tridiagonal matrix of form
\begin{equation}\label{Dn-det}
J_n=\left[
\begin{array}{ccccc}
y_0 & x_1 &  &  &  \\
z_1 & y_1 & x_2 & \\
 & z_2 & y_2 & \ddots &  \\
&  &\ddots & \ddots  &  x_n  \\
& & & z_n & y_n \\
\end{array}
\right]
\end{equation}
has nonnegative determinant if $y_0\ge 1$ and $y_k\ge x_kz_k+1$ for $1\le k\le n$.
Denote $D_{-1}:=1, D_0=y_0$ and $D_n=|J_n|$ for $n\ge 1$.
We show that $D_n\ge D_{n-1}\ge 1$ by induction on $n$.
Assume that $D_{n-1}\ge D_{n-2}\ge 1$.
Note that
$$D_n=y_nD_{n-1}-x_nz_nD_{n-2}$$
by expanding the determinant (\ref{Dn-det}) along the last row or column.
Hence $$D_n\ge y_nD_{n-1}-x_nz_nD_{n-1}=(y_n-x_nz_n)D_{n-1}\ge D_{n-1}\ge 1,$$
as desired.
Thus $J$ is TP, and so is $A$.
\end{proof}

Finally, we apply Theorem \ref{rst-thm} to two particularly interesting classes of recursive matrices,
which are introduced by Aigner in \cite{Aig99} and \cite{Aig01} respectively.
Many well-known combinatorial triangles are of such recursive matrices
(we refer the reader to Aigner~\cite{Aig99,Aig01} for more information).
The motivation of this paper is to study the total positivity of these combinatorial triangles.

\begin{coro}\label{s-thm}
Let $A=[a_{n,k}]_{n,k\ge 0}$ be an admissible matrix defined by
$$a_{0,0}=1,\quad a_{n+1,k}=a_{n,k-1}+s_ka_{n,k}+a_{n,k+1}.$$
If $s_0\ge 1$ and $s_k\ge 2$ for $k\ge 1$,
then $A$ is TP.
\end{coro}

\begin{coro}\label{st-thm}
Let $A=[a_{n,k}]_{n,k\ge 0}$ be a recursive matrix defined by
$$a_{0,0}=1,\quad a_{n+1,k}=a_{n,k-1}+s_ka_{n,k}+t_{k+1}a_{n,k+1}.$$
If $s_0\ge 1$ and $s_k\ge t_k+1$ for $k\ge 1$,
then $A$ is TP.
\end{coro}

\begin{coro}
The Bell triangle, the Catalan triangles of Aigner and Shapiro are TP respectively.
\end{coro}

\section{Concluding remarks and further work}
\hspace*{\parindent}

For two real polynomials $f(q)$ and $g(q)$ in $q$, denote $f(q)\ge_q g(q)$
if coefficients of the difference $f(q)-q(q)$ are all nonnegative.
Let $A(q)$ be an infinite matrix all whose elements are real polynomials in $q$.
It is called {\it $q$-TP} if its minors of all orders have nonnegative coefficients as polynomials in $q$.
Theorems \ref{basic-thm}, \ref{dd-thm} and \ref{rst-thm} can be carried over verbatim to their $q$-analogue.

\begin{thm}
Let $\pi=(r_k(q))_{k\ge 1}, \sigma=(s_k(q))_{k\ge 0}, \tau=(t_k(q))_{k\ge 1}$
be three sequences of polynomials in $q$ with nonnegative coefficients
and $A(q)=[a_{n,k}(q)]_{n,k\ge 0}$ be an infinite lower triangular matrix defined by
$$a_{0,0}(q)=1,\quad a_{n+1,k}(q)=r_{k}(q)a_{n,k-1}(q)+s_k(q)a_{n,k}(q)+t_{k+1}(q)a_{n,k+1}(q),$$
where $a_{n,k}(q)=0$ unless $n\ge k\ge 0$.
Then the $q$-recursive matrix $A(q)$ is $q$-TP if one of the following conditions holds:
\begin{itemize}
  \item [\rm (i)] $s_0(q)\ge_q r_1(q)$ and $s_k(q)\ge_q r_{k+1}(q)+t_{k}(q)$ for $k\ge 1$.
  \item [\rm (ii)] $s_0(q)\ge_q t_1(q)$ and $s_k(q)\ge_q r_k(q)+t_{k+1}(q)$ for $k\ge 1$.
  \item [\rm (iii)] $s_0(q)\ge_q 1$ and $s_k(q)\ge_q r_k(q)t_k(q)+1$ for $k\ge 1$.
\end{itemize}
\end{thm}

There are other forms of recursive matrices.
For example, the Eulerian triangle
$$A=[A(n,k)]_{n,k\ge 1}
=\left[\begin{array}{rrrrrr}
1 &  &  &  &  &  \\
1 & 1 &   &   &   &   \\
1 & 4 & 1 &   &   &   \\
1 & 11 & 11 & 1 &   &   \\
1 & 26 & 66 & 26 & 1 &   \\
\vdots &  & &  &  & \ddots \\
\end{array}\right],$$
where $A(n,k)$ is the Eulerian number
and satisfies the recursive relation
$$A(n+1,k)=(n-k+2)A(n,k-1)+kA(n,k).$$
Brenti suggested the following.

\begin{conj}[{\cite[Conjecture 6.10]{Bre96}}]
The Eulerian triangle $A=[A(n,k)]_{n,k\ge 1}$ is TP.
\end{conj}

The Narayana triangle
$$N=\left[N(n,k)\right]_{n,k\ge 1}
=\left[\begin{array}{rrrrrr}
1 &  &  &  &  &  \\
1 & 1 &   &   &   &   \\
1 & 3 & 1 &   &   &   \\
1 & 6 & 6 & 1 &   &   \\
1 & 10 & 20 & 10 & 1 &   \\
\vdots &  & &  &  & \ddots \\
\end{array}\right],$$
where $N(n,k)=\frac{1}{k}\binom{n-1}{k-1}\binom{n}{k-1}$ is the Narayana number
and satisfies the recursive relation
$$N(n+1,k)=\frac{n(n+1)}{2k(k-1)}N(n,k-1)+\frac{n(n+1)}{2(n-k+1)(n-k+2)}N(n,k)$$
for $k\ge 2$.
Sometimes $N$ is called the Catalan triangle
since its row sum is precisely the Catalan number:
$$\sum_{k=1}^nN(n,k)=C_n.$$
We refer the reader to Sloane's OEIS~\cite[A001263]{OEIS} for more information about the Narayana triangle.
Here we propose the following conjecture.
\begin{conj}
The Narayana triangle $N=[N(n,k)]_{n,k\ge 1}$ is TP.
\end{conj}

\section*{Acknowledgement}

This work was supported in part by the National Natural Science Foundation of China (Grant No. 11371078)
and the Specialized Research Fund for the Doctoral Program of Higher Education of China (Grant No. 20110041110039).


\end{document}